\documentclass[11pt]{amsart}
\usepackage{lipsum}% http://ctan.org/pkg/lipsum
\usepackage{amsmath, amsthm, amssymb}
\usepackage{natbib}
\usepackage{lineno,hyperref}
\usepackage[T1]{fontenc}
\usepackage[utf8]{inputenc}
\usepackage{authblk}
\makeatletter
\g@addto@macro{\endabstract}{\@setabstract}
\newcommand{\authorfootnotes}{\renewcommand\thefootnote{\@fnsymbol\c@footnote}}%
\makeatother

\newtheorem{prop}{Proposition}[section]
\newtheorem{lemma}[prop]{Lemma}

\newtheorem{theorem}[prop]{Theorem}
\newtheorem{definition}[prop]{Definition}

\newtheorem{corollary}[prop]{Corollary}
\newtheorem{proposition}[prop]{Proposition}

\newtheorem{example}[prop]{Example}

\begin{document}

\title[Stationarity Tests for Time Series]{Stationarity as a Path Property with Applications in Time Series Analysis}
\thanks{Email: yi.shen@uwaterloo.ca}
\thanks{This work is partially supported by NSERC grant.}
\renewcommand\Authands{ and }

\thanks{}

\vspace{.5ex}

\maketitle
\begin{center}

  \normalsize
  \authorfootnotes
  Yi Shen\textsuperscript{1} and Tony S. Wirjanto\textsuperscript{1,2}\par \bigskip
  \textsuperscript{1}Department of Statistics and Actuarial Science, University of Waterloo. Waterloo, ON N2L 3G1, Canada. \par
  \textsuperscript{2}School of Accounting and Finance, University of Waterloo. Waterloo, ON N2L 3G1, Canada.\par \bigskip
\thanks{Email: yi.shen@uwaterloo.ca}

\end{center}

\begin{abstract}
Traditionally stationarity refers to shift invariance of the distribution of a stochastic process. In this paper,
we rediscover stationarity as a path property instead of a distributional property. More
precisely, we characterize a set of paths denoted as $A$, which corresponds to the notion of stationarity. On one hand, the set $A$ is shown to be large enough, so that for any stationary process, almost all of its paths are in $A$. On the other hand, we prove that any path in $A$ will behave in the optimal way under any stationarity test satisfying some mild conditions. The results provide a unified framework to understand and assess the existing time series tests for stationarity, and can potentially lead to new families of stationarity tests.
\end{abstract}

\section{Motivation}
Stationarity plays an important role in time series analysis. Many statistical properties of a time series rely on the assumption that the time series is above all stationary. The tests for
stationarity, therefore, become crucial and should be applied as a preliminary step in many analysis. In the time series literature, various tests have been proposed. Many existing tests to
discriminate between stationarity and nonstationarity rely on the concept of a unit root, such as the Dickey-Fuller type tests proposed for instance by \cite{df:1979} and the KPSS type tests proposed for
instance by \cite{kspss:1992} respectively. The first type of tests has unit root as the null hypothesis, while the second type of tests has stationarity as the null hypotesis. However, this
unit-root concept is specifically defined for linear autoregressive models with finite-variance disturbances. As a result, many of the existing tests based on the unit root concept is not always suitable for
examining generic stationarity or stability property of time-series processes.

There exist a few tests based on ideas more directly related to stationarity. In the time domain, \cite{xiao:lima:2007} for instance proposed a test which works against the alternatives with
time-varying second moments. Further tests have also been developed in the frequency domain, using spectral decomposition and wavelets. To name a few, we cite the pioneering work by
\cite{priestley:rao:1969}, followed by \cite{sachs:neumann:1999} and \cite{nason:2013}. Their approach can be regarded as a mixture of the analysis in the frequency domain and in the time domain, in
the sense that they check the constancy of the result of the spectral decomposition across time. \cite{dwivedi:rao:2010} constructed a test purely in the frequency domain by considering the
correlation of the discrete Fourier transform at the canonical frequencies.

In principle, all of the tests that we cited so far are tests for second-order stationarity, also known as ``weak'' stationarity. However, the tests in the time domain can be modified to test for
strict stationarity by incorporating information from different levels. This thread of works includes \cite{kapetanios:2007}, \cite{busetti:harvey:2010}, and \cite{lima:neri:2013}. Other tests for
strict stationarity rely on more specific assumptions such as Markov property. See, for example, \cite{domowitz:elgamal:2001} and \cite{kanaya:2011}. It should also be pointed out that researchers do not always draw a clear distinction
between tests that are designed to test for strict stationarity and tests that are designed to test for second-order stationarity, due to the logical, technical and historical links between these two concepts.

Stationarity tests for time series are unique relative to their counterparts for stochastic processes in general, where a number of independent or correlated paths are often available. For time
series, typically only one path (or realization) is available, and all of the conclusions about the time series must be drawn based on the information extracted from this single path. Thus, in some
sense, stationarity tests for time series transform stationarity very naturally from a distributional property to a path property, with each particular stationary test dividing the path space into a
``stationary/acceptance region'' and a ``non-stationary/rejection region''.

A careful reader would point out that the above argument is not sufficient to transform stationarity into a path property, since the same reasoning works for all of the properties for which time series tests exist. However, there is a fundamental difference between path properties and distributional properties in terms of the results produced by the time series tests. For a path property, such as monotonicity, exceedance to a threshold, \textit{etc}., assuming that we have a large enough data set, all of the ``reasonable'' tests should give similar results to a fixed path, since there is a definite answer to the question as whether the given path possesses this property. In contrast, different tests normally give different results if the property of interest in terms of a distributional, such as Gaussian or Markov, property. In this case, the answer will depend on the test used, or more precisely, the mechanism upon which the tests are constructed.

Logically, a stationarity test for time series should capture some ``essential'' properties possessed by ``typical'' (\textit{e.g.}, almost all) paths of stationary processes, and it should be used to verify whether the
given path has this property. Equivalently, the test can also be used to verify the existence of some traits which should be essentially absent in a stationary process, and utilize this result as a basis to reject
the null hypothesis of stationarity. Following the reasoning in the last paragraph, the critical question is, what properties are deemed to be ``essential'' in distinguishing between stationarity and non-stationarity, and whether we will obtain the same result for a given path when different properties are used for evaluation?

In principle, any property which is satisfied by all of the stationary processes with a higher probability than the non-stationary processes, or the opposite case, should work. There are so many of them, so that
it seems to be hopeless to come up with a clear idea about how such a property should look like. On the other hand, interestingly, it seems that we have a relatively clear notion about which paths are ``stationary'',
or more precisely, which are not. Let us consider the following examples:

Let $\mathbf{X}=\{X_n\}_{n\in\mathbb{N}_0}$ be a time series over an infinite time horizon, where $\mathbb{N}_0$ stands for the set of all non-negative integers. Let $H$ be the path space $\mathbb
R^{\mathbb{N}_0}$ equipped with the cylindrical $\sigma$-field.

\begin{example}
If $\mathbf{x}=\{x_n\}_{n\in\mathbb{N}_0}$ is strictly increasing, then the corresponding time series should not be stationary, since $P(\mathbf X \text{ is strictly incresing})=0$ for any stationary
time series $\mathbf X$.
\end{example}

\begin{example}
If there exists $k$ such that $x_k> \sup_{i\in{\mathbb N}_0, i\neq k}x_i$, then the time series should not be stationary. Intuitively, with probability 1, a stationary time series does not have a peak
which is never attainable again.
\end{example}

Given the above examples, it might be tempting to argue that since each path is special in a certain sense, it will be rejected for stationarity by some tests. In other words, the abundance of the
criteria which can be used for stationarity will result in an empty intersection for their acceptance regions in the path space. If this is the case, then stationarity should not be considered as a path property, because it means that the result of a stationarity test for a given path solely depends on the properties upon which the test is constructed. This, however, turns out not to be the case. In fact, there exist paths which should
not be excluded from stationarity in any case, as shown by the following examples.

\begin{example}
Let $\mathbf{x}=(c,c,...)$ be a sequence of constant $c\in\mathbb{R}$. Then one should not conclude that $\mathbf x$ is not stationary. Actually, if a stationarity test rejects such a path, then for
this constant stationary process, its type I error will be identically equal to 1.
\end{example}

\begin{example}
Let $\mathbf{x}=(x_0,x_1,...)$, where $x_n=\sin(n\theta+\varphi_0), n\in\mathbb{N}_0$. This is a wave with period $2\pi/\theta$ and phase $\varphi_0$, observed at integer times. Notice that if we make
$\varphi_0$ to be random and uniformly distributed on $[0,2\pi)$, then $\mathbf{x}$ becomes a stationary process. Therefore if we consider that all of the phases are equal in determining whether the path
$\mathbf{x}$ is stationary, which seems an irrefutable argument, then such $\mathbf{x}$ should not be rejected for stationarity when tested. This example extends to all of the periodic functions observed
at integers.
\end{example}

The examples above show how a strong, intuitive distinction between stationary and non-stationary paths exists in our mind, which enables us to tell the non-stationary paths from the stationary ones
even before we venture into finding an appropriate set of criteria to discriminate them. Thus such an intuitive distinction should be built upon some principles more fundamentally than the numerous
specific path properties such as monotonicity, the number of peaks, \textit{etc}.

The goal of our paper is to flash out these principles, and to show that they actually form the basis for most existing stationarity tests. In particular, there are three conditions underlying any
stationarity test. Roughly speaking, the first condition requires that for any
event of a certain type, if it happens once, it must happen infinitely many times along the path, with a positive limiting frequency; the second condition is a mild condition which prevents any non-negligible part of the path from escaping to infinity; and the third condition is more of a technical nature, and is
related to the ergodicity of the path.

The three conditions mentioned above identify a set of paths, denoted as set $A$. We show that this is exactly the set of all of the paths which should be classified as ``stationary''. We firstly
prove that the set $A$ is large enough, such that it contains almost all of the paths of any stationary process; then we show that the set $A$ is also small enough, such that it only includes those
paths which yield the best possible results under any given stationarity test. Thus, this justifies the idea that the notion of stationarity can be transformed profitably into a path property, and that the path
space can be divided into an ``essentially stationary'' part and its complement. These results also show how the three proposed conditions can usefully serve as a basis for our intuition about the distinction between
stationarity and non-stationarity, and provide a unified framework to understand and assess the existing stationarity tests.

The rest of the paper is organized as follows. In Section 2 we introduce the basic set-up and construct the set $A$ of all the ``stationary'' paths. Section 3 shows that the set $A$ is large enough to contain almost all of the paths for any stationary process. A practical criterion to check one of the conditions that defines $A$ is also established. Finally, in Section 4 we prove that $A$ is also small enough, so that any path in $A$ will be statistically indistinguishable with a typical path of certain stationary process, in the sense that it will behave optimally under any stationarity test satisfying some mild conditions.

\section{Basic Set-up}

Let $\mathbf{x}=\{x_n\}_{n\in\mathbb{N}_0}$ be a numerical sequence in $\mathbb R$. For $k\in\mathbb{N}$, define $\mathbf{I}=I_0\times...\times I_{k-1}\in{\mathcal I}^k$, where $\mathcal I$ is the
collection of open intervals on the real line. Define a set $S_k^\mathbf{I}=S_k^\mathbf{I}(\mathbf{x})$ of non-negative integers by
$$
S_k^{\mathbf I}(\mathbf{x}):=\{n\geq 0: x_n\in I_0,..., x_{n+k-1}\in I_{k-1}\}.
$$
Denote by $N_k^\mathbf{I}=\{N_k^\mathbf{I}(n)\}_{n\in\mathbb{N}}$ the counting function of $S_k^{\mathbf I}$. That is,
$$
N_k^\mathbf{I}(n)=|S_k^{\mathbf I}\cap[0,n-1]|,
$$
where $|\cdot|$ for a set gives the number of elements in a set. We say that \textbf{Property E} holds for $\mathbf{x}$, with parameters $k$ and $\mathbf I$, if the corresponding $N_k^\mathbf{I}$
satisfies that either $N_k^\mathbf{I}(\infty)=0$, or $\lim_{n\to\infty}\frac{N_k^\mathbf{I}(n)}{n}>0$.

Define the density of a set $S\subseteq\mathbb{N}_0$ as $\lim_{n\to\infty}\frac{|S\cap[0,n-1]|}{n}$ if the limit exists. Then Property E says that $S_k^\mathbf{I}$ either is empty or has a positive
density.

Let $A_0$ be the set of all the numerical sequences such that Property E holds for all $k\in\mathbb{N}$ and ${\mathbf I}\in {\mathcal I}^k$.

We further add a tightness condition, called \textbf{Property T}:
$$
\lim_{K\to\infty}\lim_{n\to\infty}\frac{1}{n}\sum_{i=1}^n\mathbf{1}_{[0,K)}(|x_i|)=\lim_{K\to\infty}\lim_{n\to\infty}\frac{N_1^{(-K,K)}(n)}{n}=1.
$$

Intuitively, Property T prevents the ``main part'' of the sequence from escaping to infinity. We call $A_1$ a subset of $A_0$ consisting of all of the sequences in $A_0$ which satisfies Property T.

Denote by $F^1_n, n\in\mathbb{N}$ the marginal empirical measures of a sequence $\mathbf{x}\in A_1$, determined by
$$
F^1_n(I)=\frac{N_1^I(n)}{n}, \quad I\in \mathcal{I}.
$$
The fact that $\mathbf{x}\in A_0$ implies that $\lim_{n\to\infty}F^1_n(I)$ always exists, Property T then guarantees that the sequence of measures $\{F_n^1\}_{n\in\mathbb{N}}$ is tight, and hence
$\lim_{n\to\infty}F^1_n(I)$ generates a probability measure. More generally, for any $k\in \mathbb{N}$, the $k$ dimensional empirical measure $F^k_n$ is defined by
$$
F^k_n(\mathbf{I})=\frac{N_k^{\mathbf{I}}(n)}{n}, \quad \mathbf{I}\in \mathcal{I}^k.
$$
It is easy to see that Property T also assures the tightness of any finite-dimensional empirical measures, and thus $\lim_{n\to\infty}F^k_n(\mathbf{I})$ generates a probability measure on $\mathbb{R}^k$.

Together, the family of limiting probability measures $\{\lim_{n\to\infty} F^k_n\}_{k\in\mathbb N}$ satisfies the consistency condition, and thus by Kolmogorov's existence theorem, there exists a
stationary process $\mathbf{Y}=\{Y_n\}_{n\in\mathbb{N}_0}$, such that any finite dimensional distribution of $Y$:
$$
F_{Y_0,...,Y_{k-1}}=\lim_{n\to\infty}F^k_n.
$$
The process $\mathbf{Y}=\mathbf{Y}^{\mathbf x}$ is unique in distribution since all of its finite dimensional distributions are completely determined by the empirical measures of the sequence $\mathbf
x$. We call $\mathbf{Y}^{\mathbf x}$ the stationary process induced by the numerical sequence $\mathbf{x}\in A_1$.

Define set
$$
A:=\{\mathbf{x}\in A_1: \mathbf{Y}^{\mathbf{x}} \text{ is ergodic }\}.
$$

Also, notice that to make $\mathbf{Y}^{\mathbf x}$ well-defined, we only need a weaker version of Property E, where $\lim_{n\to\infty}\frac{N_k^\mathbf{I}(n)}{n}$ exists for any $k\in \mathbb N$ and
$\mathbf I\in {\mathcal I}^k$, but $N_k^\mathbf{I}(\infty)>0$ does not necessarily imply $\lim_{n\to\infty}\frac{N_k^\mathbf{I}(n)}{n}>0$.\\

\section{Coverage by $A$ of Paths from Stationary Processes}

The following theorem shows that the set $A$ is large enough, so that every stationary time series puts mass $1$ on $A$.

\begin{theorem}
Let ${\mathbf X}=\{X_n\}_{n=0,1,...}$ be a stationary time series. Then $P(\mathbf X\in A)=1$.
\end{theorem}

\begin{proof}
Firstly, by ergodic decomposition, it suffices to prove the result for the case where $\mathbf X$ is ergodic. Moreover, for ergodic process $\mathbf X$, once we prove that $P(\mathbf X\in A_0)=1$, it
follows immediately that $P(\mathbf X\in A)=1$ as well, since Property T and the ergodicity of the path are guaranteed by the pointwise ergodic theorem. Thus it suffices to prove that $P(\mathbf X\in
A_0)=1$.

The fact that Property E holds for any fixed $k$ and any single $\mathbf I$ almost surely is a trivial consequence of the pointwise ergodic theorem. As a result, Property E also holds for any
countable set of $(k,\mathbf I)$ almost surely. In the rest of the proof, for ease of notation, we will focus on the case where $k=1$, and prove that Property E holds for all $I\in \mathcal I$ almost
surely. The cases for $k>1$ follow in a similar way.

Let $F_1$ be the marginal distribution of $X_k$ for any $k=0,1,...$. Denote by $D_1$ the set of atoms of $F_1$:
$$
D_1=\{a\in \mathbb R: F_1(\{a\})>0\},
$$
and $D=D_1\cup \mathbb Q\cup\{-\infty,\infty\}$, then both $D_1$ and $D$ are at most countable sets. Hence the set
$$
A_2:=\{{\mathbf x}\in {\mathbb R}^{\mathbb N}: \text{ Property E holds for } k=1 \text{ and any } I=(a,b), a,b\in D\}
$$
satisfies $P({\mathbf X}\in A_2)=1$. Thus from now on we can assume that the paths are in $A_2$.

For any open interval $(a,b)$, there exists an increasing sequence of open intervals $\{(a_i,b_i)\}_{i=1,2,...}$, such that $a_i, b_i\in D$ for $i=1,2,...$, and
$(a,b)=\cup_i(a_i,b_i)=\lim_{i\to\infty}(a_i,b_i)$. Let the corresponding sets be $S$ and $S_i$, and the corresponding counting functions be $N(n)$ and $N_i(n)$. By construction, $S=\lim_{i\to\infty}
S_i$, and $N(n)=\lim_{i\to\infty}N_i(n)$ for $n\in\mathbb N$. Suppose $N(\infty)>0$ but $\lim_{n\to\infty}\frac{N(n)}{n}=0$ for some path in $A_2$, then for $i$ large enough, we also have
$N_i(\infty)>0$ and $\lim_{n\to\infty}\frac{N_i(n)}{n}=0$, which contradicts the construction of $A_2$. Therefore the only possibility that a path $\mathbf x$ is in $A_2\setminus A$ is that the
corresponding ratio $\frac{N(n)}{n}$ does not admit a limit as $n\to\infty$.

By the pointwise ergodic theorem, for any fixed open interval $I$, we have
$$
\frac{\sum_{i=0}^{n-1}{\mathbf 1}_{\{x_i\in I\}}}{n}\to E({\mathbf 1}_{\{X_0\in I\}})=P(X_0\in I)
$$
almost surely. Thus if we define the set
$$
B:=\{{\mathbf x}: \frac{\sum_{i=0}^{n-1}{\mathbf 1}_{\{x_i\in I\}}}{n}\to P(X_0\in I)\text{ for all } I=(a,b), a,b\in D\},
$$
then $P(B)=P(A_2\cap B)=1$. As a result, we can almost surely assume that $\mathbf x\in A_2\cap B$.

Suppose that for such an $\mathbf x$ and for an open interval $I=(a,b), a,b\in{\overline{\mathbb R}}$, the corresponding ratio $\frac{N(n)}{n}$ does not admit a limit as $n\to\infty$. Without loss of
generality, assume that $a\in D$ and $b\notin D$. The case where $a\notin D$, $b\in D$ and $a\notin D$, $b\notin D$ are similar. The non-existence of the limit implies that
$$
u:=\limsup_{n\to\infty}\frac{N(n)}{n}\neq\liminf_{n\to\infty}\frac{N(n)}{n}=:d.
$$
By definition, for any $b'\in D\cap (a,b)$,
\begin{equation}\label{ieq:lower}
\lim_{n\to\infty}\frac{\sum_{i=0}^{n-1}\mathbf{1}_{\{x_i\in(a,b')\}}}{n}\leq\liminf_{n\to\infty}\frac{N(n)}{n}=d.
\end{equation}
On the other hand, for $b''\in D\cap(b,\infty)$,
\begin{equation}\label{ieq:upper}
\lim_{n\to\infty}\frac{\sum_{i=0}^{n-1}{\mathbf 1}_{\{x_i\in(a,b'']\}}}{n}\geq \limsup_{t\to\infty}\frac{N(n)}{n}=u.
\end{equation}
The limit above exists because
$$
\sum_{i=0}^{n-1}{\mathbf 1}_{\{x_i\in(a,b'']\}}=\sum_{i=0}^{n-1}{\mathbf 1}_{\{x_i\in(a,\infty)\}}-\sum_{i=0}^{n-1}{\mathbf 1}_{\{x_i\in(b'',\infty)\}}.
$$
Subtracting (\ref{ieq:lower}) from (\ref{ieq:upper}), we have
$$
\lim_{n\to\infty}\frac{\sum_{i=0}^{n-1}{\mathbf 1}_{\{x_i\in[b',b'']\}}}{n}\geq u-d>0
$$
for any $b',b''\in D$ and $b'<b<b''$. Recall that since we work with $A_2\cap B$, this also implies that
$$
P(X_0\in[b',b''])\geq u-d.
$$
Because $D$ is dense in $\mathbb R$, we can take $b'\uparrow b$ and $b''\downarrow b$, leading to the result
$$
P(X_0=b)\geq u-d>0.
$$
However, since $b\notin D$, $b$ is not an atom of $F_1$. Thus $P(X_0=b)=0$, which is a contradiction. Hence the assumption is almost surely false and the limit exists with probability 1.
\end{proof}

In reality, checking the ergodicity of $\mathbf{Y}^{\mathbf{x}}$ for a given $\mathbf{x}$ by definition firstly requires us to fully recover the distribution of $\mathbf{Y}^{\mathbf{x}}$ from
$\mathbf{x}$, then determine whether the process $\mathbf{Y}^{\mathbf{x}}$ is ergodic given its distribution. Unfortunately none of these two steps is practical. However for a given sequence
$\mathbf{x}$, we can derive an equivalent characterization of the ergodicity, which is directly built upon the behavior of the sequence rather than the property of the measure it induces.

\begin{definition}
An \textbf{asymptotically proportional contraction} of the index set ${\mathbb N}_0$ is a subset $G$ of ${\mathbb N}_0$ consisting of disjoint intervals $G_i$ of consecutive integers:
$$
G=\cup_{i=1}^\infty G_i,
$$
satisfying
\begin{enumerate}
\item $G_i, i\in\mathbb N$ are increasingly ordered. That is, $\min\{n:n\in G_{i+1}\}>\max\{n:n\in G_i\}$, $i\in\mathbb N$; \item $|G_i|\to\infty$ as $i\to\infty$, where $|\cdot|$ is the
    number of elements (integers) in a set; \item $\frac{|[0,n-1]\cap G|}{n}\to c>0$ as $n\to\infty$.
\end{enumerate}
\end{definition}

\begin{definition}
An \textbf{asymptotically proportional contraction} of a numerical sequence $\mathbf{x}=\{x_n\}_{n=0,1,...}$ is a subsequence $\{x_{n_i}\}_{n_i\in G}$ of $\{x_n\}_{n\in{\mathbb N}_0}$, where $G$
is an asymptotically proportional contraction of the index set $\mathbb{N}_0$.
\end{definition}
Intuitively, an asymptotically proportional contraction of a numerical sequence consists of pieces of the original sequence with length of the pieces going to infinity and the fraction of coverage
converging to a fixed positive level.

\begin{theorem}\label{th:contraction}
Let $\mathbf x$ be a numerical sequence in $A_1$. Then $\mathbf x\in A$ if and only if all of its asymptotically proportional contractions induce the same process as the original sequence. That is,
for any asymptotically proportional contraction $\mathbf x'$, $k\in\mathbb N$ and $\mathbf{ I}\in\mathcal{I}^k$,
$$
\lim_{n\to\infty}\frac{{N'}_k^{\mathbf{I}}(n)}{n}=\lim_{n\to\infty}\frac{N_k^\mathbf{I}(n)}{n},
$$
where $N'$ is the counting function defined in the same way as previously but for the subsequence $\mathbf x'$.
\end{theorem}

To prove Theorem \ref{th:contraction}, let us firstly introduce the following lemma. A similar result was presented in \cite{furstenberg:1960}. However, the proof to be presented below is much
simpler, due to the difference in the framework used in this paper and that used in \cite{furstenberg:1960}, and the fact that we only need a one-directional result.

\begin{lemma}\label{l:density}
Let $\mathbf x$ be a path in $A$, therefore $\mathbf{Y}^{\mathbf x}$ be ergodic. Let $k\in\mathbb N$, $\mathbf{I}=I_0\times ...\times I_{k-1}\in{\mathcal{I}}^k$ and
$S_k^\mathbf{I}=S_k^\mathbf{I}(\mathbf x)$ be defined as previously. Then for every $\epsilon>0$, there is an $N$, such that the set
$$
R_{k,N}^\mathbf{I}:= \left\{ n\in\mathbb{N}: \left|\frac{1}{N}\sum_{i=n}^{n+N-1}\mathbf{1}_{S_k^\mathbf{I}}(i)-p_k^\mathbf{I}\right|>\epsilon \right\}
$$
has a density smaller than $\epsilon$, where the constant $p_k^\mathbf{I}=P(Y^\mathbf{x}_0\in I_0,...,Y^\mathbf{x}_{k-1}\in I_{k-1})$.
\end{lemma}

\begin{proof}
Notice that the existence of the density for the sets $R_{k,N}^\mathbf{I}$:
 $$
 \lim_{n\to\infty}\frac{\sum_{m=0}^{n-1}\mathbf{1}_{R_{k,N}^\mathbf{I}}(m)}{n}
 $$
 is guaranteed by Property E. Moreover, by the ergodicity of the path, the density of a set $R_{k,N}^\mathbf{I}$ is exactly the probability of the corresponding event, namely,
 \begin{align*}
  & \lim_{n\to\infty}\frac{\sum_{m=0}^{n-1}\mathbf{1}_{ R_{k,N}^\mathbf{I}}(m)}{n}\\
  = & P\left(\left|
  \frac{1}{N}\sum_{i=0}^{N-1}\prod_{j=0}^{k-1}\mathbf{1}_{I_j}(Y^\mathbf{x}_{i+j})-p_k^\mathbf{I}
  \right|>\epsilon\right)\\
  = & P\left(\left|\frac{1}{N}\sum_{i=0}^{N-1}\mathbf{1}_{\{\theta^i\circ \mathbf{Y}\in A_k^\mathbf{I}\}}-p_k^\mathbf{I}\right|>\epsilon\right),
 \end{align*}
where $\theta$ is the shift operator, and $A_k^\mathbf{I}$ is a subset of the path space $H$, defined as
$$
A_k^\mathbf{I}=\{\mathbf{x}\in H: x_i\in I_i, i=0,...,k-1\}.
$$

Assume that the result in Lemma \ref{l:density} is not true. Then there is $\epsilon>0$, such that for any $N\in\mathbb N$, either $ \left\{ n\in\mathbb{N}:
\frac{1}{N}\sum_{i=n}^{n+N-1}\mathbf{1}_{S_k^\mathbf{I}}(i)-p_k^\mathbf{I}>\epsilon \right\} $ or $ \left\{ n\in\mathbb{N}:
\frac{1}{N}\sum_{i=n}^{n+N-1}\mathbf{1}_{S_k^\mathbf{I}}(i)-p_k^\mathbf{I}<-\epsilon \right\} $ has a density which is greater or equal to $\frac{\epsilon}{2}$. Without loss of generality, assume that
the set
$$ \left\{ n\in\mathbb{N}: \frac{1}{N}\sum_{i=n}^{n+N-1}\mathbf{1}_{S_k^\mathbf{I}}(i)-p_k^\mathbf{I}>\epsilon \right\}
$$
has a density greater or equal to $\frac{\epsilon}{2}$ for infinitely
many $N\in \mathbb N$, denoted as $\{N_i\}_{i\in\mathbb N}$. By ergodicity of the path $\mathbf x$, this implies that
$$
P\left(\frac{1}{N_i}\sum_{j=0}^{N_i-1}\mathbf{1}_{\{\theta^j\circ \mathbf{Y}\in A_k^\mathbf{I}\}}>p_k^\mathbf{I}+\epsilon\right)\geq \frac{\epsilon}{2}
$$
for $i\in\mathbb N$. As a result, the event
$$
\left\{ \frac{1}{N}\sum_{j=0}^{N-1}\mathbf{1}_{\{\theta^j\circ \mathbf{Y}\in A_k^\mathbf{I}\}}>p_k^\mathbf{I}+\epsilon \quad\text{for infinitely many } N \right\}
$$
has a probability greater or equal to $\frac{\epsilon}{2}$. This implies that
$$
\limsup_{n\to\infty}\frac{1}{n}\sum_{j=1}^n\mathbf{1}_{\{\theta^j\circ \mathbf{Y}\in A_k^\mathbf{I}\}}\geq p_k^\mathbf{I}+\epsilon
$$
happens with a probability greater or equal to $\frac{\epsilon}{2}$.

However, since $\mathbf{Y}$ is ergodic,
$$
\lim_{n\to\infty}\frac{1}{n}\sum_{j=1}^n\mathbf{1}_{\{\theta^j\circ \mathbf{Y}\in A_k^\mathbf{I}\}}=p_k^\mathbf{I}
$$
almost surely, which is a contradiction. Therefore we conclude that the assumption is invalid and the result in Lemma \ref{l:density} holds.
\end{proof}

\begin{proof}[Proof of Theorem \ref{th:contraction}]
Assume $\mathbf{x}\in A$. For $k\in\mathbb N$, $\mathbf{I}=I_0\times...\times I_{k-1}\in\mathcal{I}^k$, define $S^\mathbf{I}_k(\mathbf{x})$ as previously. Let $\mathbf{x}'=\{x_{n_i}\}_{n_i\in G}$ be
an asymptotically proportional contraction of $\mathbf{x}$, where $G=\cup_{i}G_i$ is the corresponding asymptotically proportional contraction of $\mathbb N_0$. To prove the ``only if'' direction,
our goal is to prove that the set $S^\mathbf{I}_k(\mathbf{x}')$ has the same density as $S^\mathbf{I}_k(\mathbf{x})$. Let $c=\lim_{n\to\infty}\frac{|[0,n-1]\cap G|}{n}$. By Lemma \ref{l:density}, for
any $\epsilon>0$, there exists $N$, such that the set
$$
R_{k,N}^\mathbf{I}= \left\{ n\in\mathbb{N}_0: \left|\frac{1}{N}\sum_{j=n}^{n+N-1}\mathbf{1}_{S_k^\mathbf{I}(\mathbf{x})}(j)-p_k^\mathbf{I}\right|>\epsilon \right\}
$$
has a density smaller than $\epsilon$. Hence, the upper density of $R_{k,N}^\mathbf{I}$ in $G$, defined as
$$
\limsup_{n\to\infty}\frac{|R_{k,N}^\mathbf{I}\cap[0,n-1]\cap G|}{|[0,n-1]\cap G|},
$$
is smaller than $\frac{\epsilon}{c}$. Similar to $R_{k,N}^\mathbf{I}$, one can define
$$
{R'}_{k,N}^\mathbf{I}:=\left\{n_i\in G: \left|\frac{1}{N}\sum_{j=i}^{i+N-1}\mathbf{1}_{S_k^\mathbf{I}(\mathbf{x}')}(j)-p_k^\mathbf{I}\right|>\epsilon\right\}.
$$
Since the operation of contraction will join different segments of the original path together, ${R'}_{k,N}^\mathbf{I}$ and ${R}_{k,N}^\mathbf{I}$ will not completely agree in $G$. However, since
$\lim_{n\to\infty}|G_n|=\infty$, the two sets will have the same upper density in $G$. Therefore, the upper density of ${R'}_{k,N}^\mathbf{I}$ is also smaller than $\frac{\epsilon}{c}$. It is easy to see
that
\begin{align*}
& \limsup_{n\to\infty}\frac{1}{n}\sum_{i=0}^{n-1}\mathbf{1}_{S_k^\mathbf{I}(\mathbf{x}')}(i)\\ \leq & \limsup_{n\to\infty}\frac{|{R'}_{k,N}^\mathbf{I}\cap[0,n-1]\cap G|}{|[0,n-1]\cap G|}\cdot 1+
1\cdot (p_k^\mathbf{I}+\epsilon)\\ \leq & p_k^\mathbf{I}+\epsilon\left(1+\frac{1}{c}\right).
\end{align*}
Since $\epsilon$ can be arbitrarily small, we must have
$$
\limsup_{n\to\infty}\frac{1}{n}\sum_{i=0}^{n-1}\mathbf{1}_{S_k^\mathbf{I}(\mathbf{x}')}(i)\leq p_k^\mathbf{I}.
$$
Symmetrically, $\liminf_{n\to\infty}\frac{1}{n}\sum_{i=0}^{n-1}\mathbf{1}_{S_k^\mathbf{I}(\mathbf{x}')}(i)\geq p_k^\mathbf{I}$. Thus
$$
\lim_{n\to\infty}\frac{1}{n}\sum_{i=0}^{n-1}\mathbf{1}_{S_k^\mathbf{I}(\mathbf{x}')}(i)= p_k^\mathbf{I},
$$
which shows that $S_k^\mathbf{I}(\mathbf{x}')$ always has the same density, which is also the density of $S_k^\mathbf{I}(\mathbf{x})$.\\

Conversely, assume that $\mathbf{x}\in A_1$ but $\mathbf{x}\notin A$. Thus $\mathbf{x}$ induces a stationary process $\mathbf{Y}=\mathbf{Y}^{\mathbf{x}}$, but it is not ergodic. Therefore there exists
$p\in(0,1)$ and stationary processes $\mathbf{Z}$ and $\mathbf{W}$ with distinct distributions, such that $F_{\mathbf{Y}}=pF_{\mathbf{Z}}+(1-p)F_{\mathbf{W}}$. In particular, there exists
$k\in\mathbb N$ and $\mathbf{I}=I_0\times...\times I_{k-1}\in\mathcal{I}^k$, such that $z:=P(Z_i\in I_i, i=0,...,k-1)\neq P(W_i\in I_i, i=0,...,k-1)=:w$. Without loss of generality, assume that $z>w$.
Notice that since $\mathbf{x}$ induces $\mathbf{Y}$,
$$
\lim_{n\to\infty}\frac{|S_k^{\mathbf{I}}(\mathbf{x})\cap [0,n-1]|}{n}=P(Y_i\in I_i, i=0,...,k-1)=pz+(1-p)w.
$$

For $m\in\mathbb{N}$, define
$$
V_0:=\left\{j\in\mathbb{N}_0: \frac{|S_k^\mathbf{I}(\mathbf{x})\cap [j,j+m-1]|}{m}\geq \frac{(1+p)z+(1-p)w}{2}\right\}.
$$

Intuitively, $V_0$ is the set of the starting points of the segments of length $m$ in $\mathbf{x}$, for which the local density of the points in $S_k^\mathbf{I}(\mathbf{x})$ is higher than or equal to
$\frac{(1+p)z+(1-p)w}{2}$, which is a level between $z$ and $pz+(1-p)w$. It is clear by the construction of $A_0$ that $V_0$ has a density.

Consider process $\mathbf{Z}$. Similar to $\mathbf{x}$, we now have a random set
$$
S_k^\mathbf{I}(\mathbf{Z})=\{n\geq 0: Z_{n+i}\in I_i, i=0,...,k-1\}.
$$
Then
\begin{align*}
z=&P(Z_i\in I_i, i=0,...,k-1)\\ =&E(\mathbf{1}_{S_k^\mathbf{I}(\mathbf{Z})}(0))\\ =&E\left(\frac{1}{m}\sum_{j=0}^{m-1}\mathbf{1}_{S_k^\mathbf{I}(\mathbf{Z})}(j)\right)\\
=&E\left(\left.\frac{1}{m}\sum_{j=0}^{m-1}\mathbf{1}_{S_k^\mathbf{I}(\mathbf{Z})}(j)\right|\frac{1}{m}\sum_{j=0}^{m-1}\mathbf{1}_{S_k^\mathbf{I}(\mathbf{Z})}(j)\geq \frac{(1+p)z+(1-p)w}{2} \right)\\
&\cdot P\left(\frac{1}{m}\sum_{j=0}^{m-1}\mathbf{1}_{S_k^\mathbf{I}(\mathbf{Z})}(j)\geq \frac{(1+p)z+(1-p)w}{2}\right)\\
+&E\left(\left.\frac{1}{m}\sum_{j=0}^{m-1}\mathbf{1}_{S_k^\mathbf{I}(\mathbf{Z})}(j)\right|\frac{1}{m}\sum_{j=0}^{m-1}\mathbf{1}_{S_k^\mathbf{I}(\mathbf{Z})}(j)< \frac{(1+p)z+(1-p)w}{2} \right)\\
&\cdot P\left(\frac{1}{m}\sum_{j=0}^{m-1}\mathbf{1}_{S_k^\mathbf{I}(\mathbf{Z})}(j)< \frac{(1+p)z+(1-p)w}{2}\right)\\ \leq
&P\left(\frac{1}{m}\sum_{j=0}^{m-1}\mathbf{1}_{S_k^\mathbf{I}(\mathbf{Z})}(j)\geq \frac{(1+p)z+(1-p)w}{2}\right)\\ &+\frac{(1+p)z+(1-p)w}{2}.
\end{align*}
Hence, we have
\begin{align*}
&P\left(\frac{|S_k^\mathbf{I}(\mathbf{Z})\cap[0,m-1]|}{m}\geq \frac{(1+p)z+(1-p)w}{2}\right)\\ =&P\left(\frac{1}{m}\sum_{j=0}^{m-1}\mathbf{1}_{S_k^\mathbf{I}(\mathbf{Z})}(j)\geq
\frac{(1+p)z+(1-p)w}{2}\right)\\ \geq &\frac{(1-p)(z-w)}{2}.
\end{align*}

Since $\mathbf{Y}$ is a mixture of $\mathbf{Z}$ and $\mathbf{W}$, we obtain
$$
P\left(\frac{|S_k^\mathbf{I}(\mathbf{Y})\cap[0,m-1]|}{m}\geq \frac{(1+p)z+(1-p)w}{2}\right)\geq \frac{p(1-p)(z-w)}{2}.
$$
Then this implies that the density of the set $V_0$ is greater or equal to $\frac{p(1-p)(z-w)}{2}$, since $\mathbf{Y}$ is generated by $\mathbf{x}$. Denote the elements of $V_0$ in an increasing order
as $V_0=\{v_0,v_1,...\}$, and define a subset $V_1$ of $V_0$:
$$
V_1=\{v_{im}, i\in\mathbb{N}\}.
$$
That is, we only take each $m-$th element in $V_0$ to form $V_1$. Then $V_1$ has a density which is larger than or equal to $\frac{p(1-p)(z-w)}{2m}$. Moreover, the construction of $V_1$ guarantees
that the intervals $[j, j+m-1], j\in V_1$ are disjoint. We further take a subset of $V_1$, denoted as $V_2$, which has a density exactly equal to $\frac{p(1-p)(z-w)}{2m}$. Finally, define
$$
H=\bigcup_{j\in V_2}[j,j+m-1],
$$
then $H$ consists of disjoint intervals of integers, each with length (number of integers) $m$, and the set $H$ has density $\frac{p(1-p)(z-w)}{2}$.

Recall that $V_0$, $V_1$, $V_2$ and $H$ all depend on $m$, so we can also denote them respectively as $V_0(m)$, $V_1(m)$, $V_2(m)$ and $H(m)$. Notice, however, that the density of $H(m)$ does not depend on $m$.
Now we construct an asymptotically proportional contraction $G$ of the index set $\mathbb{N}_0$ in the following inductive way:

\begin{enumerate}
\item Define set $G(1)=H(1)$. Since $G(1)$ has a density given by $d:=\frac{p(1-p)(z-w)}{2}$, for any $\epsilon_1>0$, there exists $N(1)\in\mathbb{N}$, such that $N(1)\in G(1)$, and
$$
\left|\frac{|G(1)\cap [0,n]|}{n+1}-d\right|\leq \frac{\epsilon_1}{3}.
$$
for any $n\geq N(1)$. Moreover, since $H(2)$ also has a density given by $d$, we can take $N(1)$ large enough so that
$$
\left|\frac{|H(2)\cap [0,n]|}{n+1}-d\right|\leq \frac{\epsilon_1}{3}
$$
for any $n\geq N(1)$.

\item Let $\{\epsilon_i\}$ be a sequence of positive numbers decreasing to 0. Assume that we already have a set $G(m)$ and a positive integer $N(m)$, where $G(m)$ consists of intervals of integers
    with lengths increasing to $m$, and has a density given by $d$; $N(m)$ is the endpoint of an interval with length $m$ in $G(m)$: $N(m)-i\in G(m), i=0,...,m-1$, and satisfies
    $$
    \left|\frac{|G(m)\cap [0,n]|}{n+1}-d\right|\leq \frac{\epsilon_m}{3}
    $$
and
$$
\left|\frac{|H(m+1)\cap [0,n]|}{n+1}-d\right|\leq \frac{\epsilon_m}{3}
$$
for $n\geq N(m)$. Then define
$$
G(m+1)=(G(m)\cap[0,N(m)])\cup \bigcup_{\substack{i\in V_2(m+1),\\i\geq N(m)+1}}[i,i+m].
$$
That is, $G(m+1)$ is obtained by joining the part of $G(m)$ before $N(m)$ and the part of $H(m+1)$ after $N(m)$, but the area around the joint point is modified so that only the whole intervals in
$H(m+1)$ are kept. Notice that such a defined quantity $G(m+1)$ consists of intervals of integers with lengths increasing to $m+1$. Since both $H(m+1)$ and $H(m+2)$ has a density given by $d$,
there exists $N(m+1)>N(m)$, such that $N(m+1)-i\in G(m+1), i=0,...,m$,
$$
\left|\frac{|G(m+1)\cap [0,n]|}{n+1}-d\right|\leq \frac{\epsilon_{m+1}}{3}
$$
and
$$
\left|\frac{|H(m+2)\cap [0,n]|}{n+1}-d\right|\leq \frac{\epsilon_{m+1}}{3}.\\
$$
for $n\geq N(m+1)$.

\item Define $G$ by
\begin{align*}
G=&\lim_{m\to\infty}G(m)\\ =&\bigcup_{m=1}^\infty G(m)\cap[N(m-1)+1,N(m)]
\end{align*}
where $N(0)=-1$.\\
\end{enumerate}

The set $G$ that we constructed consists of intervals of integers with lengths going to infinity. It is not difficult to see that we can make $G$ to have a density given by $d$. Indeed, for
$m\in\mathbb N$ and any $n\in[N(m-1)+1, N(m)]$,
\begin{align*}
&\left|\frac{|G\cap[0,n]|}{n+1}-d\right|\\ =&\left|\frac{|G(m)\cap[0,n]|}{n+1}-d\right|\\ \leq & \left|\frac{|H(m)\cap[0,n]|}{n+1}-d\right|\\
&+\left|\frac{G(m-1)\cap[0,N(m-1)]}{N(m-1)+1}-\frac{H(m)\cap[0,N(m-1)]}{N(m-1)+1}\right|+O(m/n)\\ \leq &
\left|\frac{|H(m)\cap[0,n]|}{n+1}-d\right|+\left|\frac{G(m-1)\cap[0,N(m-1)]}{N(m-1)+1}-d\right|\\ &+\left|\frac{H(m)\cap[0,N(m-1)]}{N(m-1)+1}-d\right|+O(m/n)\\ \leq &
\frac{\epsilon_{m-1}}{3}+\frac{\epsilon_{m-1}}{3}+\frac{\epsilon_{m-1}}{3}+O(m/n)\\ =&\epsilon_{m-1}+O(m/n).
\end{align*}

The error term $O(m/n)$ comes from the possible difference between $H(m)$ and $G(m)$ over $[N(m-1)+1, N(m)]$ due to the modification made around the joint point, and can be made arbitrarily small by
taking $N(m-1)$ to be large enough.

As a result, $G$ is an asymptotically proportional contraction of the index set $\mathbb{N}_0$. Moreover, by construction, it is clear that the lower density of $S_k^\mathbf{I}(\mathbf{x})$ in $G$,
defined as
$$
\liminf_{n\to\infty}\frac{|S_k^\mathbf{I}(x)\cap G\cap [0,n-1]|}{|G\cap [0,n-1]|},
$$
is greater or equal to $\frac{(1+p)z+(1-p)w}{2}$. Similar as before, let $\mathbf{x}'$ be the asymptotically proportional contraction of $\mathbf{x}$ determined by $G$. Then
$S_k^\mathbf{I}(\mathbf{x}')$ will have the same limiting behavior as $S_k^\mathbf{I}(\mathbf{x})$ restricted in $G$. Hence, either $S_k^{\mathbf{I}}(\mathbf{x}')$ has a density greater or equal to
$\frac{(1+p)z+(1-p)w}{2}$, or it does not have a density, while $S_k^\mathbf{I}(\mathbf{x})$ has a density given by $pz+(1-p)w$. Thus, we have found an asymptotically proportional contraction of
$\mathbf{x}$ which does not induce the same process as the original sequence $\mathbf{x}$.
\end{proof}

\section{Results of Stationarity Tests Applied to Paths in $A$}
The previous section shows that the set of functions $A$ is large enough, such that any stationary process must put mass $1$ on $A$. In this section, our goal is to show that the set $A$ is also small
enough, in the sense that it only contains the ``essentially stationary'' paths. To this end, we consider the stationarity tests applied to the paths in $A$, and prove that the results are the
best that we can expect.

Let $T$ be a hypothesis test for sample size $n$ and consider the null hypothesis $H_0$: $\mathbf X=\{X_0,...,X_{n-1}\}$ is stationary, or more precisely, $H_0$: $\mathbf{X}$ is from a
stationary time series defined on $\mathbb{R}^{\mathbb{N}_0}$ or $\mathbb{R}^{\mathbb{Z}}$. In other words, $T$ is a mapping from $\mathbb{R}^n$ to $\{0,1\}$, where $0$ and $1$ correspond to
``acceptance'' and ``rejection'' of the null hypothesis, respectively. Alternatively, $T$ can be represented as $\mathbf{1}_{C_T}(x_0,...,x_{n-1})$, where $C_T\in\mathcal{C}_{\mathbb{R}^n}$ is the
critical region (or, equivalently, the rejection region) of the test, $\mathcal{C}_{\mathbb{R}^n}$ being the cylindrical $\sigma-$field in $\mathbb{R}^n$. Define
$$
\alpha_T(P)=P(T(\mathbf{X})=1)=P(C_T)
$$
for $P\in\mathcal{P}_0$, the collection of stationary probability measures restricted to $\mathbb{R}^n$, then the size of the test $T$ is
$$
\alpha=\sup_{P\in\mathcal{P}_0}\alpha_T(P).
$$
We further define $g_n=g_{n,0}$ to be the projection: $g_n(\mathbf{x})=(x_0,...,x_{n-1}), \mathbf{x}\in\mathbb{R}^{\mathbb{N}_0}$, and $g_{n,i}:=g_n\circ T^i$. Thus, $g_{n,i}$ is the operation of taking the moving
window of size $n$ starting from $x_i$.

\begin{theorem}\label{t:testing}
Let $\mathbf{x}\in A$. Assume that $T$ is a given test for stationarity of size $\alpha$ and with a given sample size $n$. If one of the two following conditions is satisfied:
\begin{enumerate}
\item the critical region $C_T$ is closed; or \item the boundary of the critical region: $bd(C_T)$ is a null set under any $P\in\mathcal{P}_0$,
\end{enumerate}
then the upper density of the index set
$$
\{i\in\mathbb{N}_0: g_{n,i}(\mathbf{x})\in C_T\}
$$
is smaller than or equal to $\alpha$.
\end{theorem}

Theorem \ref{t:testing} shows that if we apply a ``well-behaved'' stationarity test, in the sense that it satisfies one of the two conditions listed in the theorem, to a moving window with length $n$ of any
path $\mathbf{x}$ in the set $A$, then the limiting frequency that the null hypothesis of stationarity is rejected should not exceed the size of the test. Intuitively, this ensures that when we apply
a stationarity test to a path in $A$, we get the best possible result that we come to expect. More precisely, notice that the size $\alpha$ can be approached by the rejection rate of the null hypothesis even if
it is true. Then by the ergodic decomposition, for arbitrarily small $\epsilon>0$, there exists an ergodic process, for which the rejection rate is larger than $\alpha-\epsilon$. Interpreting
ergodicity as the equivalence between the mean across time and the mean across space, for a typical path of this ergodic process, the null hypothesis should be rejected with a limiting frequency
greater than $\alpha-\epsilon$ when the window of length $n$ moves from the origin to $+\infty$. Therefore having a limiting frequency of rejection smaller or equal to $\alpha$ is the best that
we should expect to get. Any further requirement will exclude typical paths from certain stationary processes.

The significance of Theorem \ref{t:testing} resides in the conclusion that if a path $\mathbf{x}$ is known to belong to set $A$, then it is ``statistically indistinguishable'' with a typical path from a stationary process, in the sense that its performance under any stationarity test satisfying the condition of Theorem \ref{t:testing} will be at least as good as the path from the stationary process. In other words, we should not expect to find any statistical method to be able to discriminate between $\mathbf{x}$ and a typical path from some stationary process.

\begin{proof}[Proof of Theorem \ref{t:testing}]
Let $\mathbf{x}\in A$ and $\mathbf{Y}^{\mathbf{x}}$ be the ergodic process that $\mathbf{x}$ induces. Define
$$
\mathcal{J}_n=\{J\in\mathcal{C}_{\mathbb{R}^n}: \lim_{m\to\infty}\frac{\sum_{i=0}^{m-1}\mathbf{1}_J(g_{n,i}(\mathbf{x}))}{m}=P(g_n(\mathbf{Y}^{\mathbf{x}})\in J)\},
$$
where $P$ is the stationary measure induced by $\mathbf{x}$.

By the definition of set $A$, $\mathcal J_n$ includes all of the $n$-dimensional cylinder sets (\textit{i.e.,} open hypercubes). In other words, $\mathcal{I}^n\subset \mathcal{J}_n$. Moreover, ${\mathcal J}_n$ clearly satisfies the following
properties:
\begin{enumerate}
\item $\phi\in\mathcal{J}_n, \mathbb{R}^n\in\mathcal{J}_n$; \item $J_1,J_2\in\mathcal J_n$, $J_1\supseteq J_2$ implies $J_1\setminus J_2\in\mathcal{J}_n$; \item $J_1,J_2\in\mathcal J_n$, $J_1\cap J_2=\phi$
    implies $J_1\cup J_2\in\mathcal{J}_n$.
\end{enumerate}
This is to say that $\mathcal J_n$ is closed under true difference and finite disjoint union. The following proposition is a simple consequence of the fact that the Euclidean space $\mathbb{R}^n$ with its
usual topology is complete separable.

\begin{proposition}\label{prop:jordan}
Let $C$ be a $\mathcal{C}_{\mathbb{R}^n}$-measurable set, $P$ be a probability measure on $(\mathbb{R}^n, \mathcal{C}_{\mathbb{R}^n})$. Then for any $\epsilon>0$, there exists $J\in\mathcal{J}_n$,
$J\subseteq C$, such that $P(J)\geq P(\mathring{C})-\epsilon$.
\end{proposition}

\begin{proof}

The proof of this proposition is fundamental. Here we only provide a sketch of the proof. Consider a collection of all hypercubes whose faces are parallel to the axes and whose vertices have rational
coordinates. This is a countable topological basis of $\mathbb{R}^n$ with its usual topology. Thus, for any $C$, its interior $\mathring{C}$, as an open set, can be expressed as the (countable) union
of some members of this topological basis, denoted as $B_1, B_2,...$. For any $\epsilon>0$, there exists a finite number $k(\epsilon)$, such that
$P(\cup_{i=1}^{k(\epsilon)}B_i)>P(\mathring{C})-\epsilon$. Repartitioning $\cup_{i=1}^{k(\epsilon)}B_i$ into finite disjoint hypercubes completes the proof.

\end{proof}

The proof of Theorem \ref{t:testing} becomes simple. Let $T$ be a given test of size $\alpha$ and with a sample size $n$, and let $P$ be the stationary measure induced by $\mathbf{x}$. Hence
$P(C_T)\leq \alpha$. If $T$ satisfies one of the two conditions listed in the theorem, then $P((C_T^c)^\circ)=P(C_T^c)\geq 1-\alpha$, where $(C_T^c)^\circ$ is the interior of $C_T^c$. For
$\epsilon>0$, by Proposition \ref{prop:jordan}, there exists $J\in\mathcal J_n$, $J\subseteq C_T^c$,, such that
$$
P(g_n(\mathbf{Y}^{\mathbf{x}})\in J)\geq P(C_T^c)-\epsilon\geq 1-\alpha-\epsilon.
$$
Since $J\in\mathcal J_n$, the set $\{i\in\mathbb{N}_0: g_{n,i}(\mathbf{x})\in J\}$ has a density which is greater or equal to $1-\alpha-\epsilon$. This implies that $\{i\in\mathbb{N}_0: g_{n,i}(\mathbf{x})\in
C_T^c\}$ has a lower density which is greater than or equal to $1-\alpha-\epsilon$. Since $\epsilon$ can be taken arbitrarily small, the lower density of $\{i\in\mathbb{N}_0: g_{n,i}(\mathbf{x})\in
C_T^c\}$ is at least $1-\alpha$. In other words,  the upper density of $\{i\in\mathbb{N}_0: g_{n,i}(\mathbf{x})\in C_T\}$ is smaller than or equal to $\alpha$.
\end{proof}

In practice, most of the stationarity tests introduce additional assumptions on the stochastic processes (time series) in their null hypotheses or alternative hypotheses in constructing the tests or
in analyzing their powers. A close examination of the proof of Theorem \ref{t:testing} reveals that such additional assumptions should not affect the result of the theorem. That is, if we can check
that the process $\mathbf{Y}^{\mathbf{x}}$ satisfies the additional assumptions of a test, then applying the test to a moving window of the path $\mathbf{x}\in A$ will still lead to a limiting
frequency of rejection no larger than the size of the test. Intuitively, the fact that the path $\mathbf{x}$ is in $A$ still guarantees the stationarity; if the test results in a higher frequency of
rejection, this is due to the violation of the additional assumptions rather than evidence of non-stationarity.

On the other hand, the two conditions in Theorem \ref{t:testing} are very general. As a matter of fact, a good test should have $bd(C_T)$ to be a null set under the null hypothesis after all, and this
is almost always the case in practice. Consequently, many prior studies do not even specify the openess/closedness of the critical region. It is not difficult to check that all of the stationarity
tests mentioned in Introduction satisfy the conditions of Theorem \ref{t:testing}. Thus, following our discussion on the additional assumptions, the result of Theorem \ref{t:testing} applies to all of
these tests. In some sense, what we have shown is that all of the existing time series tests for stationarity reduce to checking whether or not the given path is in the set $A$.\\

The above results are for tests with a fixed sample size. Next we discuss two types of asymptotic behaviors of paths in $A$.

The first kind of asymptotic behavior does not require any additional assumption or technical result. Many stationarity tests used in practice do not have a known exact size, but only an asymptotic size. In other words, there are sequences of tests with sample sizes $n$ increasing to infinity, such that although the size for any test with a fixed sample size is unknown, there exists a limiting size as $n\to\infty$. In this case, Theorem \ref{t:testing} immediately allows us to claim the following result.

\begin{corollary}
Let $\mathbf x\in A$. Assume that $\{T_n\}_{n\in\mathbb N}$ is a sequence of tests for stationarity, where $T_n$ is for sample size $n$ and has size $\alpha_n$. If $\lim_{n\to\infty}\alpha_n=\alpha$, and for each $n\in\mathbb N$, one of the two conditions in Theorem \ref{t:testing} is satisfied by the critical region $C_{T_n}$ of $T_n$, then for any $\epsilon>0$, there exists $N_\epsilon\in\mathbb N$, such that the upper density of the index set
$$
\{i\in \mathbb N_0: g_{n,i}(\mathbf x)\in C_{T_n}\}
$$
is smaller than $\alpha+\epsilon$ for any $n\geq N_\epsilon$.

\end{corollary}

 The second kind of asymptotic result is more challenging. For a fixed path $\mathbf x$, we apply stationarity tests to a longer and longer fraction of the path, always starting from the first term $x_0$, and look at the limiting behavior of the results of these tests. Such limiting results are typically strong and require more assumptions on the tests, as well as some more powerful technical advances. To obtain the results, it is helpful to consider the cylindrical $\sigma$-field $\mathcal C$ over the whole path space $\mathbb R^{\mathbb Z}$, and define
$$
\mathcal J=\{J\in\mathcal C: \lim_{m\to\infty}\frac{\sum_{i=0}^{m-1}\mathbf{1}_J(\theta^i(\mathbf x))}{m}=P(\mathbf Y^{\mathbf X}\in J)\},
$$
where $\theta$ is the shift operator, so that $\mathcal C$ and $\mathcal J$ do not correspond to any fixed $n$. We can improve Proposition \ref{prop:jordan} to the following result.

\begin{proposition}\label{prop:asymptotic}
Let $C$ be a $\mathcal{C}$-measurable set, and let $P$ be a probability measure on $(\mathbb R^{\mathbb Z}, \mathcal C)$. Then for any $\epsilon>0$, there exists $J\in\mathcal{J}$,
$J\subseteq C$, such that $P(J)\geq P(\mathring{C})-\epsilon$.
\end{proposition}

\begin{proof}
For each $n\in\mathbb N$, let $\mathcal J_n$ be defined as in the proof of Theorem \ref{t:testing}. Denote by $\mathcal G_n$ the collection of the sets $C\in \mathcal C_{\mathbb R^n}$ satisfying for any $\epsilon>0$, there exists $J\in\mathcal J_n$, $J\subseteq C$, such that $P(J)\geq P(C)-\epsilon$. Clearly, $\mathcal J_n\subseteq \mathcal G_n$. In particular, all of the $n$-dimensional open hypercubes are in $\mathcal G_n$. Indeed, it is not difficult to verify that all of the $n$-dimensional hypercubes, regardless of the openess/closedness of the boundaries, are all in $\mathcal J_n\subseteq \mathcal G_n$. Moreover, $\mathcal G_n$ is closed under finite disjoint unions. To see this, let $C_1,..., C_m$ be disjoint sets in $\mathcal G_n$. Let $J_1,...,J_m$ be the sets satisfying Proposition \ref{prop:jordan} for $C_1,...,C_m$ and $\epsilon_i=2^{-i}\epsilon$, $i=1,...,m$, then $J=\bigcup_{i=1}^mJ_i$ is in $\mathcal J_n$, $J\subseteq C$ and satisfies $P(J)\geq P(C)-\epsilon$. Denote by $\mathcal F_n$ the field generated by the $n$-dimensional hypercubes. Then a result in \cite{billingsley:1995} shows that each member in $\mathcal F_n$ can be expressed as a finite union of disjoint hypercubes. As a result, $\mathcal F_n\subseteq \mathcal G_n$.

Next we prove that $\mathcal F_n\subseteq \mathcal J_n$. Note that $\mathbb R^n\in \mathcal F_n\cap\mathcal J_n$, and $C\in \mathcal F_n\cap\mathcal J_n$ implies $C^c\in \mathcal F_n\cap\mathcal J_n$. Furthermore, $\mathcal F_n\cap\mathcal J_n$ is closed under union. Indeed, let $C_1, C_2\in\mathcal F_n\cap\mathcal J_n$. Then $C_1\cup C_2$ and $(C_1\cup C_2)^c$ are both in $\mathcal F_n\subseteq \mathcal G_n$. Consequently, for each $\epsilon>0$, there exist $J_{1,\epsilon},J_{2,\epsilon}\in \mathcal J_n$, $J_{1,\epsilon}\subseteq C_1\cup C_2$, $J_{2,\epsilon}\subseteq (C_1\cup C_2)^c$, such that $P(J_{1,\epsilon})\geq P(C_1\cup C_2)-\epsilon$ and $P(J_{2,\epsilon})\geq P((C_1\cup C_2)^c)-\epsilon$. Therefore, we have
\begin{align*}
& P(C_1\cup C_2)-\epsilon\\
\leq & P(J_{1,\epsilon})\\
= & \lim_{m\to\infty}\frac{\sum_{i=0}^{m-1}\mathbf{1}_{J_{1,\epsilon}}(g_{n,i}(\mathbf{x}))}{m}\\
\leq & \liminf_{m\to\infty}\frac{\sum_{i=0}^{m-1}\mathbf{1}_{C_1\cup C_2}(g_{n,i}(\mathbf{x}))}{m}.
\end{align*}
Letting $\epsilon$ to $0$ leads to the following result:
$$
\liminf_{m\to\infty}\frac{\sum_{i=0}^{m-1}\mathbf{1}_{C_1\cup C_2}(g_{n,i}(\mathbf{x}))}{m}\geq P(C_1\cup C_2).
$$
Symmetrically, using $J_{2,\epsilon}$ we have
$$
\liminf_{m\to\infty}\frac{\sum_{i=0}^{m-1}\mathbf{1}_{(C_1\cup C_2)^c}(g_{n,i}(\mathbf{x}))}{m}\geq P(C_1\cup C_2)^c.
$$
Thus, $\lim_{m\to\infty}\frac{\sum_{i=0}^{m-1}\mathbf{1}_{C_1\cup C_2}(g_{n,i}(\mathbf{x}))}{m}$ exists and is equal to $P(C_1\cup C_2)$. Hence $C_1\cup C_2\in\mathcal F_n\cap \mathcal J_n$. $\mathcal F_n\cap \mathcal J_n$ is a field. Since $\mathcal F_n$ is the field generated by the $n$-dimensional hypercubes, and all of the hypercubes are both in $\mathcal F_n$ and $\mathcal J_n$, we must have $\mathcal F_n\subseteq \mathcal J_n$.

Finally, let $\mathcal F=\bigcup_{n\in\mathbb N}\mathcal F_n$ be the field on $\mathbb R^{\mathbb Z}$ generated by all cylinder sets. Notice that since any member in $\mathcal F$ only have a finite number of finite-dimensional constraints, $\mathcal F\subseteq \bigcup_{n\in \mathbb N}\mathcal J_n\subseteq \mathcal J$. Denote by $\mathcal C'$ the collection of sets C in $\mathcal C$ satisfying for each $\epsilon>0$, there exists $J\in \mathcal F$, $J\subseteq C$, such that $P(J)\geq P(C)-\epsilon$. By definition, it is easy to see that $\mathcal C'$ contains $\phi$ and $\mathbb R^{\mathbb Z}$. Moreover, let $C_1, C_2,... \in \mathcal C'$, then for any $\epsilon>0$, there exists $N\in\mathbb N$ and $J_1, ,...,J_N\in \mathcal F$, such that $P(\bigcup_{i=N+1}^\infty C_i\setminus \bigcup_{i=1}^N C_i)\leq \frac{\epsilon}{2}$ and $P(J_i)\geq P(C_i)-2^{-i-1}\epsilon$ for $i=1,...,N$. The set $J=\bigcup_{i=1}^N J_i$ is in $\mathcal F$ and satisfies $P(J)\geq P(\bigcup_{i\in\mathbb N}C_i)-\epsilon$. Hence $\bigcup_{i\in\mathbb N}C_i\in \mathcal C'$ Similarly, it is easy to see that $\mathcal C'$ is closed under finite intersections. As a result, $\mathcal C'$ is a topology. Therefore it contains the topology generated by $\mathcal F$, which is the natural topology on $\mathbb R^{\mathbb Z}$. Thus, we can conclude that for any $\mathcal C$-measurable set $C$, for the open set $\mathring{C}$, there exists a set $J\in \mathcal F\subseteq \mathcal J$, such that $P(J)\geq P(\mathring{C})-\epsilon$.
\end{proof}

Proposition \ref{prop:jordan} and its consequence, Theorem \ref{t:testing}, show that for any time series stationarity test with a fixed sample size satisfying some mild conditions, a path in set $A$ will behave as well as a typical path from a stationary process. Proposition \ref{prop:asymptotic} allows us to generalize this statement to any asymptotic property. For instance, let $\{T_n\}_{n\in\mathbb N}$ be a sequence of stationarity tests with sample sizes $n$ and satisfying the condition in Theorem \ref{t:testing}. At the risk of abusing notations, we also use $T_n$ for the corresponding test statistics. Then for $\mathbf x\in A$, the limiting behavior of $T_n(\mathbf x)$ as $n\to\infty$ will be comparable to that of $\mathbf Y^{\mathbf x}$, which is a stationary process.

\begin{example}
If for any stationary time series $\mathbf X$, the limiting rejection rate of $T_n$
$$
\lim_{n\to\infty}\frac{\sum_{i=1}^n T_i(g_i(\mathbf X))}{n}
$$
almost surely exists and is bounded from above by a constant $\alpha$, then Proposition \ref{prop:asymptotic} implies that for any $\mathbf x\in A$ and generic $m\in\mathbb N$,
$$
\lim_{n\to\infty}\frac{\sum_{i=1}^n T_i(g_{i,m}(\mathbf x))}{n}
$$
exists and is bounded from above by $\alpha$. ``Generic'' means, the set of $m$ for which the result does not hold has a limiting density 0 in $\mathbb N$. If the assumption is relaxed to the existence of the upper/lower limit of the rejection rate and their bounds, the corresponding results holds as well for the paths in $A$.
\end{example}

\textbf{Acknowledgements.} This work is partially supported by NSERC grant No. 469065. The authors would like to thank Brett Coburn for being a research assistant for this project.


\begin{thebibliography}{}

\bibitem[Billingsley (1995)]{billingsley:1995}{\sc Billingsley, P.} (1995): {\em Probability and Measure\/}.
\newblock John Wiley \& Sons.

\bibitem[Busetti and Harvey (2010)]{busetti:harvey:2010}{\sc Busetti, F. and Harvey, A.C.} (2010): Tests of strict stationarity based on quantile indicators. {\em Journal of Time Series Analysis\/},
    31, 435-450.

\bibitem[Dickey and Fuller (1979)]{df:1979}{\sc Dickey, D.A. and Fuller, W.A.} (1979): Distribution of the estimators for autoregressive time series with a unit root. {\em Journal of the American
    Statistical Association\/}, 74, 427-431.

\bibitem[Domowitz and El-Gamal (2001)]{domowitz:elgamal:2001}{\sc Domowitz, I. and El-Gamal, M.A.} (2001): A consistent nonparametric test of ergodicity for time series with applications. {\em Journal of Econometrics\/}, 102, 365-398.

\bibitem[Dwivedi and Subba Rao (2011)]{dwivedi:rao:2010}{\sc Dwivedi, Y. and Subba Rao, S.} (2011): A test for second order stationarity of a time series based on the Discrete Fourier Transform. {\em
    Journal of Time Series Analysis\/}, 32, 68-91.

\bibitem[Furstenberg (1960)]{furstenberg:1960}{\sc Furstenberg, H.} (1960): {\em Stationary Processes and Prediction Theory\/}.
\newblock Princeton Univerisity Press, Princeton.

\bibitem[Kanaya (2011)]{kanaya:2011}{\sc Kanaya, S.} (2011): A non-parametric test for stationarity in continuous-time Markov processes. Job Market Paper, University of Oxford.

\bibitem[Kapetanios (2007)]{kapetanios:2007}{\sc Kapetanios, G.} (2007): Testing for Strict Stationarity. Working Paper 602, Queen Mary, Univerisity of London.

\bibitem[Kwiatkowski, Phillips, Schmidt and Shin (1992)]{kspss:1992}{\sc Kwiatkowski, D., Phillips, P.C.B., Schmidt, P. and Shin, Y.} (1992):  Testing the null hypothesis of stationarity against the
    alternative of a unit root. {\em Journal of Econometrics\/}, 54, 159-178.

\bibitem[Lima and Neri (2013)]{lima:neri:2013}{\sc Lima, L.R. and Neri, B.} (2013): A test for strict stationarity. In {\em Uncertainty Analysis in Econometrics with Applications}, VanNam Huynh et al.
    Eds., Series Advances in Intelligent Systems and Computing, Springer-Verlag.

\bibitem[Nason (2013)]{nason:2013}{\sc Nason, G.} (2013): A test for second-order stationarity and approximate confidence intervals for localized autocovariances for locally stationary time series.
    {\em Journal of the Royal Statistical Society. Series B\/}, 75, 879-904.

\bibitem[Priestley and Subba Rao (1969)]{priestley:rao:1969}{\sc Priestley, M.B. and Subba Rao, T.} (1969): A test for non-stationarity of a time series. {\em Journal of the Royal Statistical Society.
    Series B\/}, 31, 140-149.

\bibitem[von Sachs and Neumann (1999)]{sachs:neumann:1999}{\sc von Sachs, R. and Neumann, M.H.} (1999): A wavelet-based test for stationarity. {\em Journal of Time Series Analysis\/}, 21, 597-613.

\bibitem[Xiao and Lima (2007)]{xiao:lima:2007}{\sc Xiao, Z. and Lima, L.R.} (2007): Testing covariance stationarity. {\em Econometric Reviews\/}, 26(6), 643-667.


\end{thebibliography}
\end{document}